\documentclass[10pt]{scrartcl}
\usepackage{amsmath, amssymb, amsthm}
\usepackage{graphicx}
\usepackage{subfigure}
\usepackage{color}
\usepackage{cite}

\newcommand{\be}{\begin{equation}}
\newcommand{\ee}{\end{equation}}

\theoremstyle{plain}
\newtheorem{thm}{Theorem}[section]

\newtheorem{prop}{Proposition}[section]
\newtheorem{cor}[prop]{Corollary}
\newtheorem{lem}[prop]{Lemma}

\newtheorem{rem}[prop]{Remark}
\newtheorem{definition}{Definition}[section]

\graphicspath{{images/}}

\author{Undine Leopold and Horst Martini}
\title{Monge points, Euler lines, and Feuerbach spheres in Minkowski spaces}
\begin{document}
\maketitle

\begin{center}
\textbf{Dedicated to Egon Schulte and K\'aroly Bezdek on the occasion of their 60th birthdays}
\end{center}

\begin{abstract}
\textbf{Abstract:} It is surprising, but an established fact that the field of Elementary Geometry referring to normed spaces (= Minkowski spaces) is not a systematically developed discipline. There are many natural notions and problems of elementary and classical geometry that were never investigated in this more general framework, although their Euclidean subcases are well known and this extended viewpoint is promising. An example is the geometry of simplices in non-Euclidean normed spaces; not many papers in this direction exist. Inspired by this lack of natural results on Minkowskian simplices, we present a collection of new results as non-Euclidean generalizations of well-known fundamental properties of Euclidean simplices. These results refer to Minkowskian analogues of notions like Euler line, orthocentricity, Monge point, and Feuerbach sphere of a simplex in a normed space. In addition, we derive some related results on polygons (instead of triangles) in normed planes.\\[.05ex]
\end{abstract}

\section{Introduction}\label{sec:intro}

Looking at basic literature on the geometry of finite dimensional real Banach spaces (see, e.g., the monograph \cite{t1996:mg} and the surveys \cite{msw2001:tgomsas1} and \cite{ms2004:tgomsas2}), the reader will observe that there is no systematic representation of results in the spirit of elementary and classical geometry in such spaces (in other words, the field of Elementary Geometry is not really developed in normed spaces, also called Minkowski spaces).  This is not only meant in the sense that a classifying, hierarchical structure of theorems is missing. 
Also, it is already appealing to find the way of correctly defining analogous notions. An example of such a non-developed partial field is the geometry of simplices in non-Euclidean Minkowski spaces. Inspired by this indicated lack of natural results on Minkowskian simplices, we derive a collection of new results which reflect non-Euclidean analogues and extensions of well known properties of Euclidean simplices. These results are based on, or refer to, generalizations of notions like Euler lines, orthocentricity (of course depending on a suitable orthogonality notion), Monge points, and Feuerbach spheres of simplices in Minkowski spaces. It should be noticed that some of these topics are even not established for Minkowski planes; most of our results are derived immediately for simplices in Minkowski spaces of arbitrary finite dimension.

In plane Euclidean geometry, the \emph{Euler line} of a given triangle is a well-studied object which contains many interesting points besides the circumcenter and the centroid of this triangle. Other special points on the Euler line include the orthocenter and the center of the so-called \emph{nine-point-} or \emph{Feuerbach circle}. Notions like this can be extended to simplices in higher dimensional Euclidean space, and the respective results can sometimes be sharpened for important subfamilies of general simplices, like for example the family of orthocentric simplices. Using new methods developed by Grassmann for studying the $d$-dimensional Euclidean space, this was done already in the 19th century. Two early related references are \cite{m1884:aeesudedarvbvd} and \cite{r1901:oeotpoto}. Deeper results were obtained later; the concept of Euler line and some related notions have been generalized to Euclidean higher dimensional space in \cite{bb2004:aotnpcfons, ehm2005:osatc, ehm2008:osab,e1940:oos,e1950:otfsoaos,hm2013:osattgot,h1951:udfkmos,k1964:ooas,m1960:udegufkdnds,i1962:tpicitnds} for orthocentric simplices, and in \cite{bb2005:tmpat3psoans,ehm2005:coscarfs,f1976:hfns,m1962:aoasins,m1960:udegufkdnds,s1981:aagotel} for general simplices. 
Other interesting generalizations in Euclidean geometry refer to Euler lines of cyclic polygons, see \cite{hg2014:tfcatorcotcp}. For a few results in Minkowski planes and spaces we refer to \cite{at2004:oirnsatel,ag1960:otgomp,cg1985:tfpims,ms2007:tfcaoinp,mw2009:oosiscnp}. The Feuerbach circle of a triangle in the Euclidean plane passes through the feet of the three altitudes, the midpoints of the three sides, and the midpoints of the segments from the three vertices to the orthocenter of that triangle. Beautiful generalizations of the Feuerbach circle to $d$-dimensional Euclidean space for orthocentric simplices have been obtained in \cite{bb2004:aotnpcfons,e1950:otfsoaos,h1951:udfkmos,i1962:tpicitnds}, and for general simplices in \cite{bb2005:tmpat3psoans, m1960:udegufkdnds, f1976:hfns}. Minkowskian analogues have so far only been discussed in normed planes, see \cite{ag1960:otgomp,ms2007:tfcaoinp, rspr2015:oosimp}.

A $d$-dimensional (\emph{normed} or) \emph{Minkowski space} $(\mathbb{R}^d,\|\cdot\|)$ is the vector space $\mathbb{R}^d$ equipped with a norm $\| \cdot \|$. A norm can be given implicitly by its \emph{unit ball} $B(O,1)$, which is a convex body centered at the origin $O$; its boundary $S(O,1)$ is the \emph{unit sphere} of the normed space. Any homothet of the unit ball is called a \emph{Minkowskian ball} and denoted by $B(X,r)$, where $X$ is its center and $r>0$ its radius; its boundary is the \emph{Minkowskian sphere} $S(X,r)$. Two-dimensional Minkowski spaces are \emph{Minkowski planes}, and for an overview on what has been done in the geometry of normed planes and spaces we refer to the book \cite{t1996:mg}, and to the surveys \cite{msw2001:tgomsas1} and \cite{ms2004:tgomsas2}.

The fundamental difference between non-Euclidean Minkowski spaces and the Euclidean space is the absence of an inner product, and thus the notions of angles and orthogonality do not exist in the usual sense. Nevertheless, several \emph{types of orthogonality} can be defined (see \cite{ab1988:oinlsas1},\cite{ab1989:oinlsas2}, and \cite{amw2012:oboaioinls} for an overview), with \emph{isosceles} and \emph{Birkhoff orthogonalities} being the most prominent examples. We say that $y$ is \emph{isosceles orthogonal} to $x$, denoted $x\perp_I y$, when $\|x+y\|=\|x-y\|$. Isosceles orthogonality is thus the orthogonality of diagonals in a parallelogram with equal side lengths (a rhombus in Euclidean space). It is also the orthogonality of chords over a diameter. By contrast, $y$ is \emph{Birkhoff orthogonal} to $x$, denoted $x\perp_B y$, when $\|x\|\leq \|x+\alpha y\|$ for any $\alpha \in \mathbb{R}$. Thus Birkhoff orthogonality is the (unsymmetric) orthogonality of a radius $x$ and corresponding tangent vector $y$ of some ball centered at the origin $O$. For hyperplanes and lines, there is the notion of \emph{normality}. A direction (vector) $v$ is \emph{normal} to a hyperplane $E$ if there exists a radius $r>0$, such that $E$ supports the ball $B(O,r)$ at a multiple of $v$. Equivalently, $v$ is normal to $E$ if any vector parallel to $E$ is Birkhoff orthogonal to $v$.

For any two distinct points $P$, $Q$, we denote by $[PQ]$ the \emph{closed segment}, by $\langle PQ \rangle$ the \emph{spanned line} (affine hull), and by $[PQ\rangle$ the \emph{ray} $\{P+\lambda (Q-P)\ \vert\ \lambda \geq 0\}$; we write $\|[PQ]\|$ for the \emph{length} of $[PQ]$. We will use the usual abbreviations {\rm aff}, {\rm conv}, $\partial$, and {\rm cone} for the affine hull, convex hull, boundary and cone over a set, respectively.

\begin{figure}
\begin{center}
\includegraphics[width=.7\textwidth]{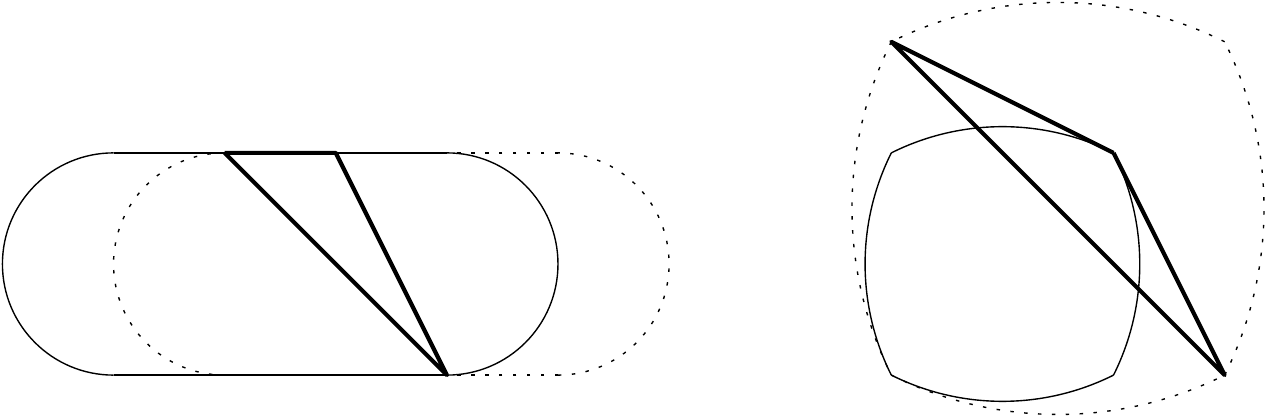}
\end{center}
\caption{A triangle with several circumcenters (left), and a triangle without a circumcenter (right), as illustrated by suitable homothets of the unit ball.}
\label{fig:circumcenters0}
\end{figure}

In this article, we focus on the geometry of simplices in $d$-dimensional Minkowski spaces. As usual, a \emph{$d$-simplex} is the convex hull of $d+1$ points in general linear position, or the non-empty intersection of $d+1$ closed half-spaces in general position. We underline that by \emph{circumcenters of simplices} we mean the \emph{centers of circumspheres} (or \emph{-balls}) \emph{of simplices}, i.e., of Minkowskian spheres containing all the vertices of the respective simplex (see, e.g., \cite{ams2012:medcacinpp1}). A related, but different notion is that of \emph{minimal enclosing spheres of simplices}, sometimes also called circumspheres (cf., e.g., \cite{ams2012:medcacinpp2}); this notion is not discussed here. In the two-dimensional situation, circumspheres and -balls are called \emph{circumcircles} and \emph{-discs}.
In Minkowski spaces, simplices may have several, precisely one, or no circumcenter at all, depending on the shape of the unit ball, see Figure \ref{fig:circumcenters0}. Examples without circumcenters may only be constructed for non-smooth norms, as all smooth norms allow inscription into a ball \cite{g1969:sioah,m1987:tdoamispocg}. \emph{We focus on the case where there is at least one circumcenter.}

\section{Orthocentric simplices and the Monge point in Euclidean space}\label{sec:euclidean}

We begin with a short survey on results related to orthocentricity in Euclidean space. In Euclidean geometry, not every simplex in dimension $d\geq 3$ possesses an orthocenter, i.e., a point common to all the altitudes. However, if such a point $H$ exists, the simplex is called \emph{orthocentric} and possesses a number of special properties (compare the survey contained in \cite{ehm2005:osatc} and \cite{hm2013:osattgot}). The following proposition is well known (see again \cite{ehm2005:osatc}).

\begin{prop}\label{prop:orthocentricity1}
A $d$-simplex $T$ in Euclidean space is orthocentric if and only if the direction of every edge is perpendicular to the affine hull of the vertices not in that edge (i.e., the affine hull of the opposite $(d-2)$-face). Equivalently, a $d$-simplex in Euclidean space is orthocentric if and only if any two disjoint edges are perpendicular. 
\end{prop}

The $(d-2)$-faces of a $d$-polytope are sometimes called \emph{ridges}, see \cite{ms2002:arp}. 
The following fact (see also the survey in \cite{ehm2005:osatc}) can be proved in many ways, and has been posed as a problem in the American Mathematical Monthly \cite{khm1998:ouo}. Note that orthocenters are not defined for an edge or a point.
\begin{prop}
In an orthocentric Euclidean $d$-simplex ($d\geq 3$), the foot of every altitude is the orthocenter of the opposite facet. 
\end{prop}

In absence of a guaranteed orthocenter, the literature on Euclidean geometry (e.g. \cite{c2003:gmatmpott, ac1964:mpsg} for three dimensions, \cite{bb2005:tmpat3psoans,ehm2005:osatc,hm2013:osattgot} for the general case) defines the \emph{Monge point} of a tetrahedron or higher-dimensional simplex as the intersection of so-called \emph{Monge (hyper-)planes}. The Monge point coincides with the Euclidean orthocenter if the latter exists \cite{ac1964:mpsg,c2003:gmatmpott, bb2005:tmpat3psoans}. From this, theorems about the Euler line, the Feuerbach circle, etc. can be generalized to higher dimensional simplices, see all the references given in the Introduction, and see Section \ref{sec:euler} for Minkowskian analogues. We recall the definition and the following theorems from \cite{bb2005:tmpat3psoans}. 
\begin{definition}
Let $T$ be a $d$-simplex in Euclidean $d$-space. A \emph{Monge hyperplane} is a hyperplane which is perpendicular to an edge of the simplex and which passes through the centroid of the opposite $(d-2)$-face (ridge).
\end{definition}

\begin{thm} (Monge Theorem)
The Monge hyperplanes of a Euclidean $d$-simplex have precisely one point in common, which is called the \emph{Monge point} $N$ of the simplex.
\end{thm}

\begin{thm} (Orthocenter Theorem)
In an orthocentric Euclidean $d$-simplex, the Monge point $N$ coincides with the orthocenter $H$.
\end{thm}

\begin{thm}\label{thm:mannheim}(Mannheim Theorem, see \cite{ac1964:mpsg,c2003:gmatmpott} for $d=3$, and \cite{ bb2005:tmpat3psoans} for arbitrary $d$)
For any $d$-simplex, the $d+1$ planes, each determined by an altitude of a $d$-simplex and the Monge point of the corresponding facet, pass through the Monge point of the $d$-simplex.
\end{thm} 

Regular simplices are orthocentric. Regular simplices are also \emph{equilateral}, i.e., all their edges have equal length, as well as \emph{equifacetal}, which means that all their facets are isometric (congruent). Furthermore, the circumcenter $M$, centroid $G$, orthocenter $H$, and \emph{incenter} I, i.e., the center of the unique inscribed sphere touching all facets, coincide. Conversely, we have the following statement, see \cite{ehm2005:osatc}. 
\begin{thm}\label{thm:regularity}
A Euclidean $d$-simplex $T$ is regular, if and only if any of the following conditions are fulfilled:
\begin{enumerate}
\item $T$ is equilateral.
\item $T$ is orthocentric and any two of the centers $M$, $G$, $I$, $H$ coincide. 
\item $T$ is orthocentric and equifacetal.
\end{enumerate}
\end{thm}

As we will see in the next Section, the concept of Monge point generalizes to arbitrary Minkowski spaces, at least for simplices with a circumcenter. 
\section{The Monge point of simplices in Minkowski spaces}\label{sec:monge}

In this section, we generalize the definition of Monge point and Monge hyperplanes to Minkowski spaces of arbitrary (finite) dimension $d\geq 2$. 

\begin{definition}
Let $(\mathbb{R}^d,\|\cdot\|)$ be a $d$-dimensional Minkowski space, and let $T$ be a $d$-simplex with a circumcenter $M$. For each pair $(F,E_F)$ of a ridge $F$ and opposite edge $E_F$, and if $M$ is not the midpoint of $E_F$, define the \emph{associated Monge line} as the line through the centroid of $F$ which is parallel to the line through $M$ and the midpoint of $E_F$.
\end{definition}

\begin{thm}\label{thm:monge1}
Let $(\mathbb{R}^d,\|\cdot\|)$ be a $d$-dimensional Minkowski space, and let $T$ be a $d$-simplex with a circumcenter $M$. 
Then the Monge lines of $T$ are concurrent in a single point $N_M$, called the \emph{Monge point of $T$}. 
\end{thm}

Before proving the theorem, essentially following the outline for Euclidean space in \cite{bb2005:tmpat3psoans}, we first define a \emph{quasi-median} of a $d$-simplex as a line joining the centroid of a $(d-2)$-face of the simplex with the midpoint of the opposite edge. The following Lemma was proved in \cite{bb2005:tmpat3psoans} for Euclidean space, yet due to the definition of the centroid it holds true in any Minkowski space.

\begin{lem}
The quasi-medians of a $d$-simplex $T$ intersect in its centroid. The centroid divides each quasi-median in the ratio $2\colon (d-1)$ (with the segment measuring $\frac{d-1}{d+1}$ of the length of the quasimedian ending in the midpoint of an edge).
\end{lem}

\begin{proof} (Proof of the theorem)
The proof is similar to, but more general than, the one in \cite{bb2005:tmpat3psoans} for Euclidean space. First, for each $(d-2)$-face $F$ denote its centroid $G(F)$, and let $G(E_F)$ be the midpoint or centroid of the opposite edge $E_F$. Since a $d$-simplex possesses ${d+1 \choose 2}$ edges (ridges) and $M$ can be located at the midpoint of at most one of them, the auxiliary lines $\langle MG(E_F)\rangle$ are well-defined for at least ${d+1 \choose 2}-1$ pairs $(F,E_F)$. The auxiliary line $\langle MG(E_F)\rangle$, if well-defined, is parallel to the associated Monge line $\langle G(F)L(F)\rangle$ of $(F,E_F)$, where we define $L(F):=G(F)+G(E_F)-M$. Second, if $M=G$, then $G$ and $G(F)$ both lie on $\langle MG(E_F)\rangle$, i.e., each auxiliary line coincides with the associated Monge line, and all these lines intersect in $N_M:=M=G$ (and this is the only point, since different edge midpoints define different lines $\langle MG(E_F)\rangle$). If $M\neq G$, then auxiliary line and Monge line are distinct. Observe that each quasimedian $[G(F)G(E_F)]$ connects a Monge line and the corresponding auxiliary line. 
\begin{figure}
\begin{center}
\includegraphics[width=\textwidth]{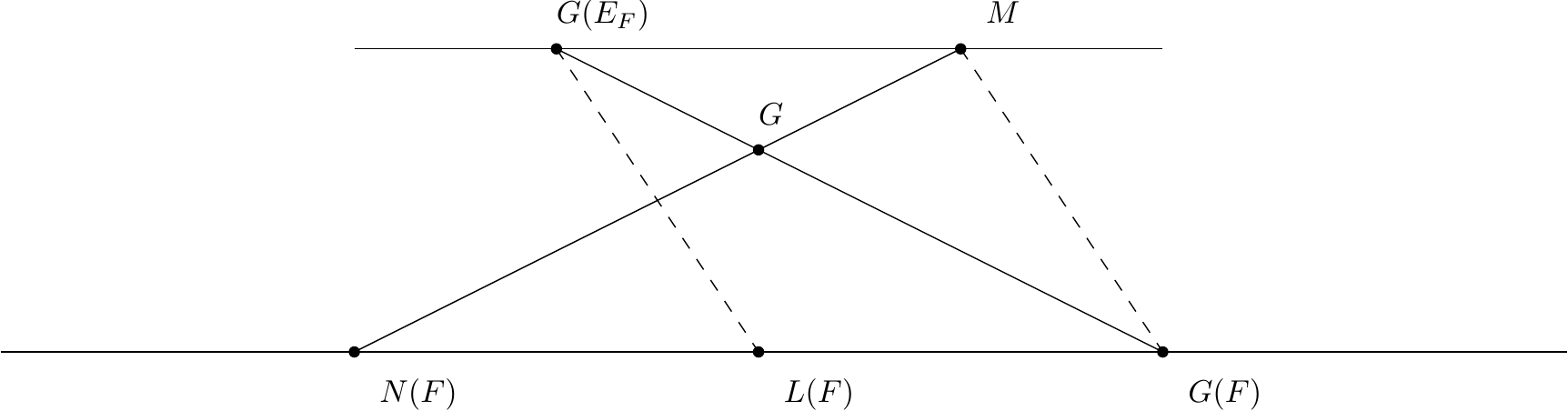}
\end{center}
\caption{Location of the Monge point.}
\label{fig:mongepoint1}
\end{figure}
The centroid $G$ of the simplex $T$ divides each quasimedian in the ratio $2\colon (d-1)$, so the same division ratio holds true for the segment $[MN(F)]$ which passes through the given circumcenter $M$, the centroid $G$ of $T$, and ends at the point $N(F)$ on $[G(F)L(F)\rangle$, see Figure \ref{fig:mongepoint1}. As a consequence of this common ratio, all points $N(F)$ are indeed the same point $N_M$, solely dependent on the chosen circumcenter (and the given simplex), and all rays $[G(F)L(F)\rangle$ meet at $N_M$. 
\end{proof}

In keeping with the tradition in Euclidean space, we want to reformulate the theorem in terms of hyperplanes.

\begin{definition}
Let $(\mathbb{R}^d,\|\cdot\|)$ be a $d$-dimensional Minkowski space, and let $T$ be a $d$-simplex with a circumcenter $M$. Suppose $M$ is not the midpoint of an edge $E_F$ opposite a $(d-2)$-face $F$ of the simplex. For the pair $(F,E_F)$ define the \emph{auxiliary pencil} of hyperplanes through $M$ and the midpoint of $E_F$. Furthermore, define the associated \emph{Monge hyperplane pencil} for the pair $(F,E_F)$ as the translate of the auxiliary pencil such that all hyperplanes go through the centroid of $F$. 
\end{definition}

\begin{cor}\label{cor:monge1}
Let $(\mathbb{R}^d,\|\cdot\|)$ be a $d$-dimensional Minkowski space, and let $T$ be a $d$-simplex with a circumcenter $M$. 
Then the hyperplanes of all (well-defined) Monge hyperplane pencils of $T$ intersect in a single point, namely the Monge point of $T$. 
\end{cor}

The following corollary tells us the precise location of the Monge point with respect to the vertices of the simplex and the given circumcenter.

\begin{cor}
Let $T={\rm conv}\{A_0,\ldots,A_d\}$ be a $d$-simplex in $d$-dimensional Minkowski space, possessing a circumcenter $M$. Then the associated Monge point is determined as 

\[N_M=M+\frac{\sum_{i=0}^d (A_i-M)}{d-1}.\]
\end{cor}

\begin{proof}
Let $F$ be a ridge of the simplex, opposite the edge $E_F$, such that $G(E_F)\neq M$ (i.e., the edge midpoint is distinct from $M$; such an edge must exist). From the proof of Theorem \ref{thm:monge1} we deduce for $M\neq G$ that \[\|[MG(E_F)]\|\colon\|[G(F)N_M]\| = \|[MG]\|\colon\|[GN_M]\|= (d-1)\colon 2.\]
Thus
\[N_M=M+(d+1)\frac{G-M}{d-1}=M+\frac{(d+1)\frac{\sum_{i=0}^{d}(A_i - M)}{d+1}}{d-1}=M+\frac{\sum_{i=0}^d (A_i-M)}{d-1}.\]
For $M=G$ we obtain $N_M=M=G$.
\end{proof}

\begin{rem}
In Euclidean context, each \emph{Monge hyperplane} passes through the centroid of a $(d-2)$-face $F$ and is perpendicular to the opposite edge $E_F$ (here \emph{opposite edge} means the edge between the two vertices not in the ridge $F$). However, we see that perpendicularity is not necessary, and any hyperplane containing the associated Monge line as per our definition is suitable (provided the Monge line is well-defined). Therefore, while our Minkowskian Monge pencils contain the correct Monge hyperplanes in Euclidean context, we have the confirmation that \emph{orthogonality of lines and hyperplanes} need not necessarily play a role when finding the Monge point. The concept of Monge point is even an \emph{affine} concept, as the circumcenter property of $M$ is used nowhere (i.e., \emph{any} point $M$ can be used to construct ``Monge lines'' intersecting at $N_M$ with the analytical expression given above). 
\end{rem}

In particular, we obtain the following corollary, which appears to be new also for the Euclidean case.

\begin{cor}\label{cor:monge}
Let $(\mathbb{R}^d,\|\cdot\|)$ be a $d$-dimensional Minkowski space, and let $T$ be a $d$-simplex with a circumcenter $M$. For each ridge $F$ and the opposite edge $E_F$ with midpoint $G(E_F)$, if $M\neq G(E_F)$ and $\langle MG(E_F)\rangle$ is \emph{not} parallel to $F$, define an $M$-hyperplane as the hyperplane containing $F$ and being parallel to $\langle MG(E_F)\rangle$. Then all defined $M$-hyperplanes intersect in the Monge point $N_M$.  
\end{cor}

\begin{proof}
Let $A_0,\ldots,A_d$ denote the vertices of $T$. Observe that, since the medial hyperplanes of $T$ are in general position, $M$ lies in at most $d$ of the $d+1$ medial hyperplanes. Without loss of generality, $M$ does not lie in the medial hyperplane between $A_0$ and its opposite facet. Since $G([A_0A_i])$ lies in that medial hyperplane for $i=1,\ldots,d$, and the ridge $F_{0,i}$ opposite $[A_0A_i]$ is parallel to that medial hyperplane, we conclude that $[MG([A_0A_i])]$ is not parallel to $F_{0,i}$, and the $M$-hyperplanes are defined at least for the $d$ pairs $(F_{0,i},[A_0A_i])$.

Consider the $(d-1)$-simplex \[T_0:={\rm conv} \{G([A_0A_i]), i=1,\ldots,d\},\] which is a homothet of the \emph{facet} $F_0$ of $T$ opposite $A_0$ with homothety center $A_0$ and factor $\frac{1}{2}$. The related $(d-1)$-simplex $T_0'$ 
is obtained by homothety of $T_0$ in $G(T_0)=\frac{A_0+G(F_0)}{2}$ and with homothety factor $-(d-1)$. Observe that the $(d-2)$-dimensional facets of $T_0'$ (ridges of $T$) pass through the vertices of $T_0$ and are parallel to the $(d-2)$-dimensional facets of $T_0$. 

Now, the $d$ $M$-hyperplanes previously considered are parallel to the hyperplanes defined by the facets of the $d$-simplex \[{\rm conv} \{M \cup T_0'\}\] through the vertex $M$. Therefore, these $M$-hyperplanes are in general position, intersecting only in the Monge point $N_M$ which, by definition, is contained in every defined $M$-hyperplane. 

\end{proof}

Another theorem concerning the Monge point in Euclidean space is the Mannheim theorem, see \cite{c2003:gmatmpott} for the three-dimensional case and \cite{bb2005:tmpat3psoans} for generalizations. It is our Theorem \ref{thm:mannheim} above, and it presents an example of a statement that cannot be extended to Minkowski spaces. The simple reason is that hyperplane sections of Minkowskian balls need not be centrally symmetric. Therefore, in general the concept of Monge point of a $d$-simplex cannot be transferred to its facets.

\section{Euler lines and generalized Feuerbach spheres of Minkowskian simplices}\label{sec:euler}

We define as \emph{Euler line associated to a circumcenter $M$} the straight line connecting $M$ with the centroid $G$. Thus, in the case of the centroid being a circumcenter, the associated Euler line is not well-defined. 
We now consider the situation in $d$-dimensional Minkowski space for $d\geq 2$.  

\begin{definition}\label{malt}
For a $d$-simplex $T:={\rm conv}\{A_0,\ldots,A_d\}$ with circumcenter $M$, define the \emph{complementary line of a facet with respect to $M$} as the translate of the line between the circumcenter $M$ of the simplex and the centroid of the facet, passing through the opposite vertex. If $A_1,\ldots,A_{d}$ are the vertices of the chosen facet with centroid $G_{0}$, then the complementary line is $A_0+t\cdot(G_{0}-M), t\in \mathbb{R}$.
\end{definition}

\begin{rem}
As in the planar case, for smooth norms such a circumcenter always exists (see \cite{m1992:isoacb}, and \cite[\S 7.1]{msw2001:tgomsas1}). For a non-smooth norm, simplices without a circumcenter may exist (see again Figure \ref{fig:circumcenters0} (right) for the planar situation, and it is easy to construct examples also for general $d$).
\end{rem}

The following theorem is an easy consequence of the definition of the centroid.

\begin{thm}\label{cpoint}
The complementary lines of the facets of a $d$-simplex $T$ with respect to a fixed circumcenter $M$ connect all the vertices to the same point, the \emph{complementary point} $P_M$ associated to $M$.
\end{thm}

\begin{proof}
Let $T={\rm conv}\{A_0,\ldots,A_d\}$, and let $G_j$ denote the centroid of the facet opposite vertex $A_j$. Then the point
\[P_M=M+\sum_{i=0}^{d} (A_i-M)= A_j + d\left(\frac{\displaystyle\sum_{\substack{i=0\\i\neq j}}^{d}A_i}{d}-M\right) = A_j+d(G_j-M) \textnormal{ for each } j=0,\ldots d,\]
lies on each complementary line. For each $j=0,\ldots,d$, we have $\|[P_MA_j]\|=d\|MG_j\|$. 
\end{proof}

Various useful types of orthogonalities have been defined in Minkowski spaces for pairs of vectors, all coinciding with the usual orthogonality in Euclidean space, yet we only have normality as a concept for vectors and (hyper-)planes. We call each segment $[P_MA_j]$ on a complementary line the \emph{complementary segment} associated to the opposite facet. As such, a complementary segment is not orthogonal to a hyperplane in any known sense. However, in dimension two we obtain the familiar isosceles orthogonality between an edge of a simplex (triangle side) and the corresponding complementary segment (orthogonality if we are in the Euclidean plane!), and the complementary point is the \emph{$C$-orthocenter} \cite{ag1960:otgomp,ms2007:tfcaoinp}. Unlike the $C$-orthocenter, the notion of complementary point generalizes to any higher dimension. 

\begin{rem}
The complementary point is even an affine notion, as we only used division ratios of segments on a line. The point $P_M$ can be constructed for any point $M$ (circumcenter or not) in the following way: take the line connecting $M$ to the centroid of a simplex facet (if distinct from $M$), and then consider the translated line passing through the vertex opposite the chosen facet. All lines of the latter kind intersect in a point (denoted $P_M$ in the present article), which has already been observed by Snapper \cite{s1981:aagotel}.
\end{rem}

The complementary point 
and Monge point associated to a simplex with circumcenter $M$ possess the following properties. 

\begin{thm}\label{thm:eulerline}
Let $T$ be a $d$-simplex ($d\geq 2$) in Minkowskian space $(\mathbb{R}^d,\|\cdot\|)$, with a circumcenter $M$ distinct from its centroid $G$.\\
(a) The associated complementary point $P_M$ and the Monge point $N_M$ lie on the Euler line $\langle MG\rangle$.\\ 
(b)  The centroid $G$ divides the segment $[MP_M]$ internally in the ratio $1:d$.\\
(c) The associated Monge point $N_M$ divides the segment $[MP_M]$ internally in the ratio $1:(d-2)$.\\
(d) The centroid $G$ divides the segment $[MN_M]$ internally in the ratio $(d-1):2$. 
\end{thm}

\begin{proof}
Let $T={\rm conv}\{A_0,\ldots,A_d\}$.
That the Euler line $\langle GM\rangle$ associated to $M$ passes through $N_M$ and $P_M$ can be seen from the following equations:
\begin{eqnarray*}
G&=&\frac{\sum_{i=0}^{d} A_i}{d+1} = M+\frac{\sum_{i=0}^{d} (A_i-M)}{d+1},\\
N_M&=&M+\frac{\sum_{i=0}^{d} (A_i-M)}{d-1},\\
P_M&=&M+\sum_{i=0}^{d} (A_i-M).
\end{eqnarray*} 
Thus (a) is proved. The above equations also immediately prove (b) and (c). 
Proving (d) is an easy exercise in arithmetic:
\begin{eqnarray*}
\|G-M\|\colon\|N_M-G\|&=&\left\|\frac{\sum_{i=0}^{d} (A_i-M)}{d+1}\right\|\colon \left\|\frac{\sum_{i=0}^{d} (A_i-M)}{d-1}-\frac{\sum_{i=0}^{d} (A_i-M)}{d+1}\right\|\\
&=&\left\|\frac{\sum_{i=0}^{d} (A_i-M)}{d+1}\right\|\colon \left\|\frac{2 \sum_{i=0}^{d} (A_i-M)}{(d-1)(d+1)}\right\|\\
&=&(d-1)\colon 2.
\end{eqnarray*}
\end{proof} 

\begin{rem}
We see that $N_M$ can be obtained by homothety of $M$ from center $G$, with homothety ratio $-\frac{2}{d-1}$. Moreover, recall the $M$-hyperplanes from Corollary \ref{cor:monge} which intersect in $N_M$. The above homothety takes each $M$-hyperplane to a certain parallel hyperplane through $M$. It turns out that these \emph{central planes} (through the circumcenter $M$) encompass the supporting hyperplanes through $M$ of the auxiliary simplex ${\rm conv} \{M \cup T_0'\}$ in the proof of Corollary \ref{cor:monge}.
\end{rem}

Considering the points of interest in Theorem \ref{thm:eulerline}, one may ask whether the point $M+\frac{\sum_{i=0}^{d} (A_i-M)}{d}$ on the Euler line, dividing $[MP_M]$ internally in the ratio $1:(d-1)$, holds any special meaning. It turns out that it is the center of a sphere analogous to the well-known  Feuerbach circle of a triangle in the Euclidean plane. 
The extension to higher dimensional normed spaces for the case $M\neq G$ is as follows (for the ``degenerate'' case $M=G$ we refer to Corollary \ref{cor:coincidence}).

\begin{thm}\label{thm:feuerbach}(The $2(d+1)$- or Feuerbach sphere of a $d$-simplex)
In an arbitrary Minkowski $d$-space, let $T={\rm conv}\{A_0,\ldots,A_{d}\}$ be a $d$-simplex with a circumcenter $M$ and circumradius $R$, and let $G (\neq M)$ be its centroid. The sphere with center $F_M:=M+\frac{\sum{(A_i-M)}_{i=1}^{d+1}}{d}$ on the Euler line and of radius $r:=\frac{R}{d}$ passes through the following $2(d+1)$ points:\\
(a) the centroids $G_i$, $i=0,\ldots,d$, of the facets $F_i$ of $T$ ($F_i$ is opposite vertex $A_i$), and\\
(b) the points $L^M_{i}$ dividing the segments connecting the Monge point $N_M$ to the vertices $A_i$ of $T$, $i=0,\ldots,d$, in the ratio $1:(d-1)$. \\
Moreover, $S(F_M,r)$ is a homothet of the circumsphere $S(M,R)$ with respect to the centroid $G$ and homothety ratio $\frac{-1}{d}$, i.e., $G$ divides the segment $[F_MM]$ internally in the ratio $1:d$, and $F_M$ divides the segment $[N_MM]$ internally in the ratio $1:(d-1)$.
\end{thm}

\begin{rem}
In analogy to the Feuerbach circle in the plane centered at the nine-point-center, we call $F_M$ the \emph{$2(d+1)$-center} of the simplex with respect to the circumcenter $M$, and $S(F_M, \frac{R}{d})$ its \emph{Feuerbach} or \emph{$2(d+1)$-sphere}.
\end{rem}

\begin{proof}
The centroid of a facet opposite vertex $A_j$ is $G_j=\frac{\displaystyle\sum_{\substack{i=0\\i\neq j}}^{d}A_i}{d}$. We have $R=\|A_j-M\|$ for any $j=0,\ldots, d$, and thus 
\[\|G_j-F_M\|=\left\|\frac{\sum_{\substack{i=0\\i\neq j}}^{d}A_i}{d}-M-\frac{\sum_{i=0}^{d}(A_i-M)}{d}\right\|=\left\|\frac{M-A_j}{d}\right\|=\frac{R}{d},\] which proves that $S(F_M,\frac{R}{d})$ passes through the points in (a).

The Monge point is $N_M=M+\frac{\sum_{i=0}^{d} (A_i-M)}{d-1}$, thus 

\begin{eqnarray*} L^M_{j}&:=&M+\frac{\sum_{i=0}^{d} (A_i-M)}{d-1}+\frac{A_j-M-\frac{\sum_{i=0}^{d} (A_i-M)}{d-1}}{d}\\
&=& M+\frac{(d-1)\sum_{i=0}^{d} (A_i-M)}{d(d-1)}-\frac{M-A_j}{d}\\
&=& M+\frac{\sum_{i=0}^{d} (A_i-M)}{d}-\frac{M-A_j}{d}.
\end{eqnarray*}

Therefore,
\begin{eqnarray*}
\|L^M_j-F_M\|&=&\left\|M+\frac{\sum_{i=0}^{d} (A_i-M)}{d}-\frac{M-A_j}{d}-M-\frac{\sum_{i=0}^{d}(A_i-M)}{d}\right\|\\
&=&\left\|-\frac{M-A_j}{d}\right\|=\frac{R}{d},
\end{eqnarray*}

which proves that $S(F_M,\frac{R}{d})$ passes through the points in (b). We also have 

\begin{eqnarray*}
\|F_M-G\|\colon\|G-M\|&=&\frac{\left\|M+\frac{\sum_{i=0}^{d}(A_i-M)}{d}-M-\frac{\sum_{i=0}^{d}(A_i-M)}{d+1}\right\|}{\left\|M+\frac{\sum_{i=0}^{d}(A_i-M)}{d+1}-M\right\|}\\
&=&\left\|\frac{\sum_{i=0}^{d}(A_i-M)}{d(d+1)}\right\|\colon \left\|\frac{\sum_{i=0}^{d}(A_i-M)}{d+1}\right\|=1\colon d
\end{eqnarray*}

and

\begin{eqnarray*}
\|N_M-F_M\|\colon\|F_M-M\|&=&\frac{\left\|M+\frac{\sum_{i=0}^{d}(A_i-M)}{d-1}-M-\frac{\sum_{i=0}^{d}(A_i-M)}{d}\right\|}{\left\|M+\frac{\sum_{i=0}^{d}(A_i-M)}{d}-M\right\|}\\
&=&\left\|\frac{\sum_{i=0}^{d}(A_i-M)}{d(d-1)}\right\|\colon \left\|\frac{\sum_{i=0}^{d}(A_i-M)}{d}\right\|=1\colon (d-1),
\end{eqnarray*}

proving the remaining statements. 
\end{proof}

\begin{figure}
\begin{center}
\includegraphics[width=\textwidth]{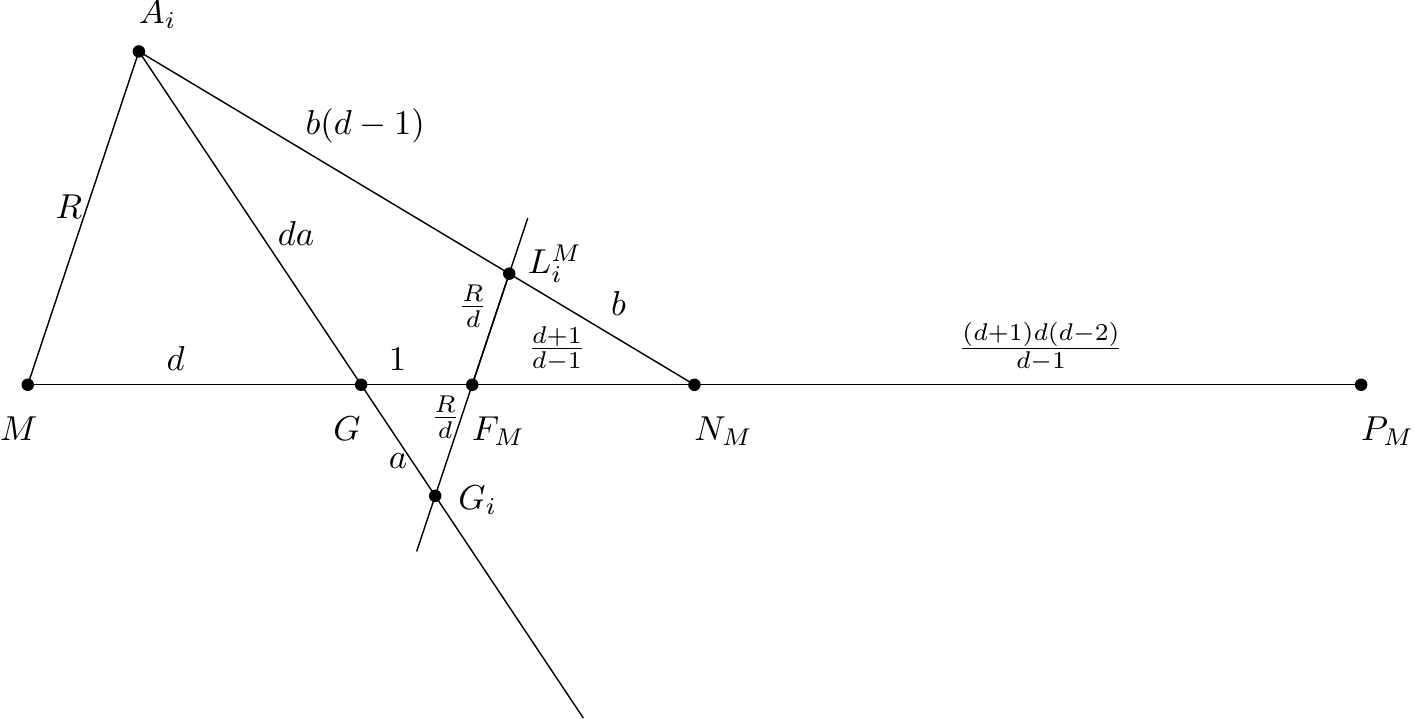}
\end{center}
\caption{Points on the Euler line and Feuerbach sphere, and ratios of line segments.}
\label{fig:euler1}
\end{figure}

\begin{rem}
As noted in the Introduction, the sphere construction has been done for the Euclidean case in several earlier works, giving a $3(d+1)$-sphere. In Minkowski space, we ''lose'' the $(d+1)$ points which are orthogonal projections of the $L^M_i$ onto the facets $F_i$. In the planar case, this has already been pointed out in \cite{ag1960:otgomp, ms2007:tfcaoinp}.
\end{rem}
\begin{rem}\label{rem:spatial}
Consider the $(d+1)$-dimensional spatial representation of this configuration where the segments between $M$ and the vertices of our simplex are projections of some segments spanning a $(d+1)$-dimensional parallelepiped. Then, the Euler line $\langle MP_M \rangle$ corresponds to the projection of the main diagonal of the parallelepiped, and the points dividing the main diagonal in the ratio $1:d$, $1:(d-1)$, and $1:(d-2)$ project to the centroid, the center of the Feuerbach-$2(d+1)$-sphere, and the Monge point, respectively. 
\end{rem}

Since it can be shown that $N_M$ divides the segment $[F_MM]$ externally in the ratio $1:d$, i.e., $[F_MM]$ is divided harmonically by $G$ and $N_M$, we obtain the following corollary, the second statement of which has been noted in \cite{bb2004:aotnpcfons} for Euclidean orthocentric simplices and the orthocenter. For a strictly convex normed plane ($d=2$), the second statement can be found in \cite[Theorem 4.6]{ms2007:tfcaoinp}. 

\begin{cor}
The Monge point $N_M$ associated to a circumcenter $M$ of a $d$-simplex $T$ is the center of homothety between the Feuerbach-$2(d+1)$-sphere centered at $F_M$ and the circumsphere centered at $M$, with homothety ratio $1:d$. For any line from $N_M$ meeting the associated circumsphere of $T$ in $Q$, the point $P$ dividing $[N_MQ]$ internally in the ratio $1:(d-1)$ is located on the Feuerbach sphere; conversely, for any line from $N_M$ meeting the associated Feuerbach sphere in $P$, the point $Q$ dividing the segment $[N_MP]$ externally in the ratio $d:(d-1)$ is located on the circumsphere of $T$.
\end{cor}

We conclude this Section with a Corollary which is an immediate consequence of the affine nature of Theorem \ref{thm:eulerline}. 
\begin{cor}\label{cor:coincidence}
In a $d$-simplex in Minkowskian $d$-space, the points $M$, $G$, $F_M$, $N_M$, $P_M$ are either collinear (on the Euler line), or they all coincide. 
\end{cor}
In the latter case, instead of speaking of the Euler line not being well-defined, sometimes the term \emph{collapsing Euler line} is used.

\section{Generalizations for polygons in the plane}\label{sec:polygons}

Generalizations of the concept of Euler line and Feuerbach circle have not just focused on raising the dimension of the space; there have also been attempts to generalize to polygons. We will now see that easy generalizations arise if we consider such polygons as projections of higher-dimensional simplices or sections of parallelepipeds. This relates to \emph{descriptive geometry} (see also Remark \ref{rem:spatial}).

B. Herrera G\'omez \cite{hg2014:tfcatorcotcp} and S.N. Collings \cite{c1967:cpatel} have written about \emph{remarkable circles} in connection with \emph{cyclic polygons} in the Euclidean plane. Their definition of \emph{cyclic polygon} as a polygon possessing a circumcircle is directly extendable to any normed plane. Necessarily, cyclic polygons are convex.  

Let $P={\rm conv}\{A_0,\ldots,A_d\}$, $d\geq 3$, be a cyclic polygon with circumcenter $M$ in the normed plane $(\mathbb{R}^2,\|\cdot\|)$. We may view the vertices of $P$ as the images under projection of certain vertices of a $(d+1)$-dimensional parallelepiped $Q$ in $(d+1)$-dimensional space to an affine plane (which we then endow with the norm $\|\cdot\|$), namely the vertices adjacent to $M'$ where $M'$ projects to $M$ (compare Remark \ref{rem:spatial}). This makes $P$ the projection of that hyperplane section $P'$ of $Q$ which is defined by all the vertices adjacent to $M'$. Alternatively, we may view $P$ as the shadow of a $d$-simplex $T$, which itself is a projection of the hyperplane section $P'$ of $Q$ to an affine $(d-1)$-subspace.

We now define the points $P_M$ (\emph{complementary point}), $N_M$ (\emph{Monge point}), $G$ (\emph{centroid}), $F_M$ (\emph{$2(d+1)$-center}) of the polygon to be the respective parallel projections of the following distinguished points on the main diagonal of the parallelepiped, which would have the corresponding meaning for the $d$-simplex $T$ when $M'$ projects to a circumcenter of $T$, see Section \ref{sec:euler}. That is, 

\begin{eqnarray*}
G&=&\frac{\sum_{i=0}^{d} A_i}{d+1} = M+\frac{\sum_{i=0}^{d} (A_i-M)}{d+1} \quad \textnormal{ is called the \emph{centroid} of the polygon } P,\\
F_M&=&M+\frac{\sum_{i=0}^{d} (A_i-M)}{d} \quad \textnormal{ is called the \emph{$2(d+1)$-center} of the polygon } P,\\
N_M&=&M+\frac{\sum_{i=0}^{d} (A_i-M)}{d-1} \quad \textnormal{ is called the \emph{Monge point} of the polygon } P,\\
P_M&=&M+\sum_{i=0}^{d} (A_i-M) \quad \textnormal{ is called the \emph{complementary point} of the polygon } P.
\end{eqnarray*}

These points either coincide or are collinear on the \emph{Euler line} of the polygon $P$ (compare Corollary \ref{cor:coincidence}), with the division ratios given in Theorem \ref{thm:eulerline}. We can then easily deduce the following relationships.

\begin{thm}\label{thm:feuerbach_poly1}
Let $P={\rm conv}\{A_0,\ldots,A_d\}$, $d\geq 3$ be a cyclic polygon with circumcenter $M$ and circumradius $R$ in the normed plane $(\mathbb{R}^2,\|\cdot\|)$.
Then:\\
(a) The complementary point $P_M$ is common to all the circles $S(P_M^{i},R)$, $i=0,\ldots,d$, where $P_M^{i}$ is the complementary point of the subpolygon $P_i={\rm conv}\left(\{A_0,\ldots,A_d\}\setminus\{A_i\}\right)$ with respect to the circumcenter $M$.\\
(b) The lines $\langle A_iP_M^{i} \rangle$ are concurrent in $C_M$, where $C_M:=M+\frac{1}{2}\sum_{i=0}^d (A_i-M)$ is the midpoint of $[MP_M]$ and called the \emph{spatial center} of $P$ with respect to $M$. \\
(c) The midpoints $E_i$ of the segments joining the vertices $A_i$, $i=0,\ldots,d$, with the complementary point $P_M$ are concyclic in the circle $S(C_M,\frac{R}{2})$. \\
(d) The point $C_M$ is common to all the circles $S(C_M^{i},\frac{R}{2})$, where $C_M^{i}$ is the spatial center of the subpolygon $P_i$ with respect to the circumcenter $M$, $i=0,\ldots,d$, and the points $C_M^{i}$ also lie on the circle $S(C_M,\frac{R}{2})$.
\end{thm}

\begin{proof}
We have \[P_M=M+\sum_{j=0}^d (A_j-M)=M+\sum_{\substack{j=0\\j\neq i}}^d (A_j-M)+(A_i-M)=P_M^{i}+(A_i-M).\]
Since $(A_i-M)$ is a radius of any translate of the circle $S(M,R)$, we obtain the statement in (a).
In the spatial representation in $(d+1)$-dimensional space, the vertex projecting to the complementary point $P_M$ is the endpoint opposite $M'$ of the main diagonal of the parallelepiped $Q$ (i.e., the line which projects to the Euler line), whereas the pre-images of the points $P_M^{i}$ are vertices adjacent to the pre-image of $P_M$. Thus the pre-images of each point $P_M^{i}$ and $A_i$, $i=0,\ldots,d$, together span another main diagonal of the parallelepiped $Q$. The main diagonals of the parallelepiped intersect in one point $C'$ (the centroid of the parallelepiped), and this point halves each main diagonal. The projection of this point is the point $C_M$ by definition, which proves part (b). Note that at most $d-1$ of the lines $\langle A_iP_M^{i} \rangle$ may not be well-defined, and precisely when their pre-images are parallel to the null space of the projection, but at least $2$ lines remain to determine the point $C_M$.   Part (d) is similar to part (a), in that
\[C_M=M+\frac{1}{2}\sum_{j=0}^d (A_j-M)=M+\frac{1}{2}\sum_{\substack{j=0\\j\neq i}}^d (A_j-M)+\frac{1}{2}(A_i-M)=C_M^{i}+\frac{1}{2}(A_i-M).\] The second statement in (d) follows trivially.
Finally, for part (c), consider Figure \ref{fig:euler1} and observe that the line $\langle C_M E_i\rangle$ is parallel to $\langle MA_i\rangle$ for each $i=0,\ldots,d$.
\end{proof}

\begin{rem}\label{rem:feuerbach1} For the Euclidean plane and $d=3$, part (d) is well known \cite[pp. 22--23]{y1968:gt2}. For all $d\geq 3$, the statements (a)--(c) have been established in \cite{hg2014:tfcatorcotcp} where $P_M$ is called the orthocenter, and $S(C_M,\frac{R}{2})$ the Feuerbach circle of the polygon. For strictly convex normed planes part (d) has been shown in \cite[Theorem 4.18]{ms2007:tfcaoinp}, calling the point $C_M$ the center of the Feuerbach circle $S(C_M,\frac{R}{2})$, and the circles $S(C_M^{i},\frac{R}{2})$ the Feuerbach circles of the subpolygons. The motivation in either case was to observe a radius half as long as the radius of the original circumcircle. We see that the statements extend in some way to \emph{all} Minkowski planes, though one has to be careful in their formulation; recall that in planes which are not strictly convex, we cannot necessarily speak of \emph{the (unique)} circumcircle, or \emph{the (unique)} intersection of several circles.
\end{rem}

\begin{rem}
Note that $M$ is a circumcenter of $P$, and also a circumcenter for each of its sub-polygons with $d\geq 3$ vertices. The analogous statement for a $d$-simplex in $d$-space is wrong, i.e., a circumcenter of a $d$-simplex $T$ is \emph{not} a circumcenter for each of its facets, which is the reason for the lack of analogous higher-dimensional statements involving the complementary points of facets of $T$ in Section \ref{sec:euler}.
\end{rem}

An alternative, equally plausible definition of (orthocenter and) Feuerbach circle of a polygon in the Euclidean plane was given by Collings \cite{c1967:cpatel}. This, too, generalizes to normed (Minkowski) planes, and is easily provable using the spatial representation given above. Both concepts of Feuerbach circles are illustrated in Figure \ref{fig:polygon}, for cyclic pentagons in the $\ell_1$-norm.

\begin{thm}\label{thm:feuerbach_poly2}
Let $P={\rm conv}\{A_0,\ldots,A_d\}$, $d\geq 3$, be a cyclic polygon with circumcenter $M$ and circumradius $R$ in the normed plane $(\mathbb{R}^2,\|\cdot\|)$.\\
(a) The Monge point $N_M$ is the point of intersection of the lines $\langle A_i N_M^{i}\rangle$, $i=0,\ldots,d$, where $N_M^{i}$ is the Monge point of the subpolygon $P_i={\rm conv}\left(\{A_0,\ldots,A_d\}\setminus\{A_i\}\right)$.\\
(b) The centroids $G_i$ of the subpolygons $P_i={\rm conv}\left(\{A_0,\ldots,A_d\}\setminus\{A_i\}\right)$, $i=0,\ldots,d$, are concyclic on $S(F_M,\frac{R}{d})$, where $F_M$ is the $2(d+1)$-center of the polygon. Furthermore, the circle $S(F_M,\frac{R}{d})$ passes through the $(d+1)$ points $L^{M}_i$ dividing the segments $[N_MA_i]$ in the ratio $1\colon (d-1)$.\\ 
(c) The Monge points $N_M^{i}$ of the subpolygons are concyclic on the circle \[S\left(M+\frac{1}{d-2}\sum_{j=0}^d (A_j-M),\frac{R}{d-2}\right)\] with its center on the Euler line.\\ 
\end{thm}

\begin{proof}
We have 
\begin{eqnarray*}N_M=M+\frac{1}{d-1}\sum_{j=0}^d (A_j-M)&=&A_i+\frac{d-2}{d-1}\left(\left(M+\frac{1}{d-2}\sum_{\substack{j=0\\j\neq i}}^d (A_j-M)\right) -A_i\right)\\
&=& A_i+\frac{d-2}{d-1}\left(N_M^{i} -A_i\right),\end{eqnarray*}
which proves part (a). Part (b) is clear with Theorem \ref{thm:feuerbach} and the fact that the segments $[F_M G_i]$ and $[F_M L_i^{M}]$ have equal length and are homothets of $[MA_i]$ for each $i=0,\ldots,d$ (with factor $\frac{1}{d}$ and homothety center $N_M$).  For part (c), observe that for each $i=0,\ldots,d$, $N_M^{i}$ is the intersection of the lines $\langle MG_i\rangle$ and $\langle A_iN_M\rangle$, see also Figure \ref{fig:euler1}. Since the above equation shows that $N_M$ divides the segment $[A_iN_M^{i}]$ in the ratio $(d-2)\colon 1$, the homothet of the circumsphere with respect to homothety center $N_M$ and homothety ratio $-\frac{1}{d-2}$ passes through the $N_M^{i}$. Thus the corresponding center can be calculated as $M+\frac{1}{d-2}\sum_{j=0}^d (A_j-M)$ (on the Euler line), and the radius is $\frac{R}{d-2}$. 
\end{proof}

\begin{figure}
\begin{center}
\subfigure[The Feuerbach circle of half size.]{
\includegraphics[width=.6\textwidth]{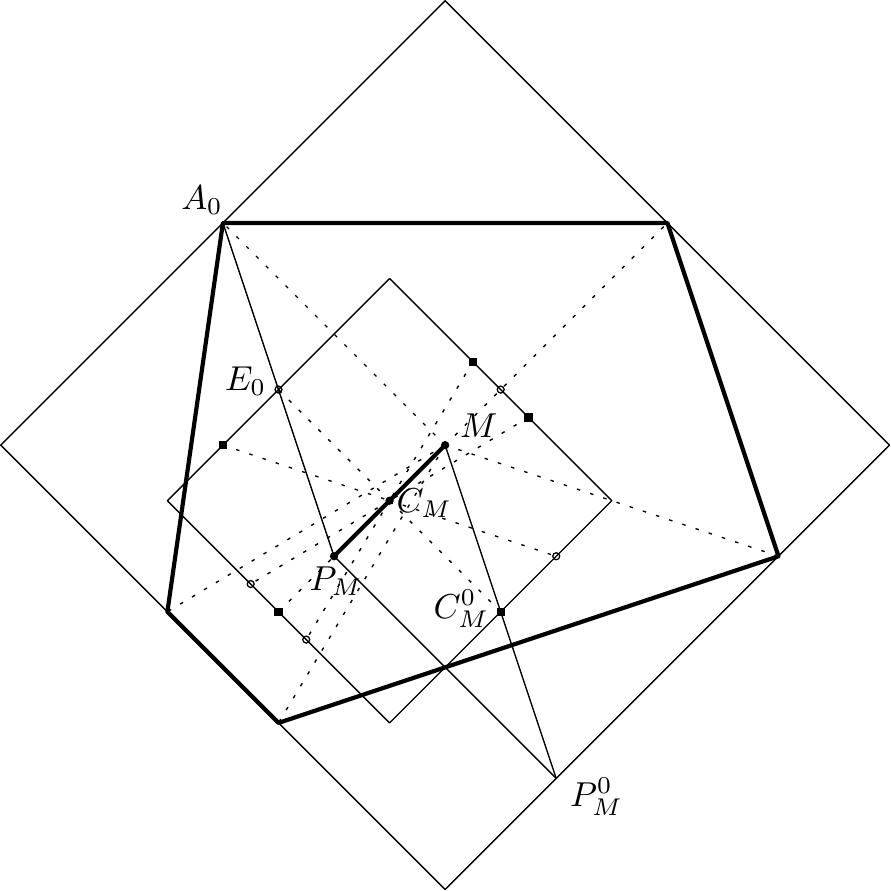}}
\subfigure[The Feuerbach circle of Collings.]{
\includegraphics[width=.6\textwidth]{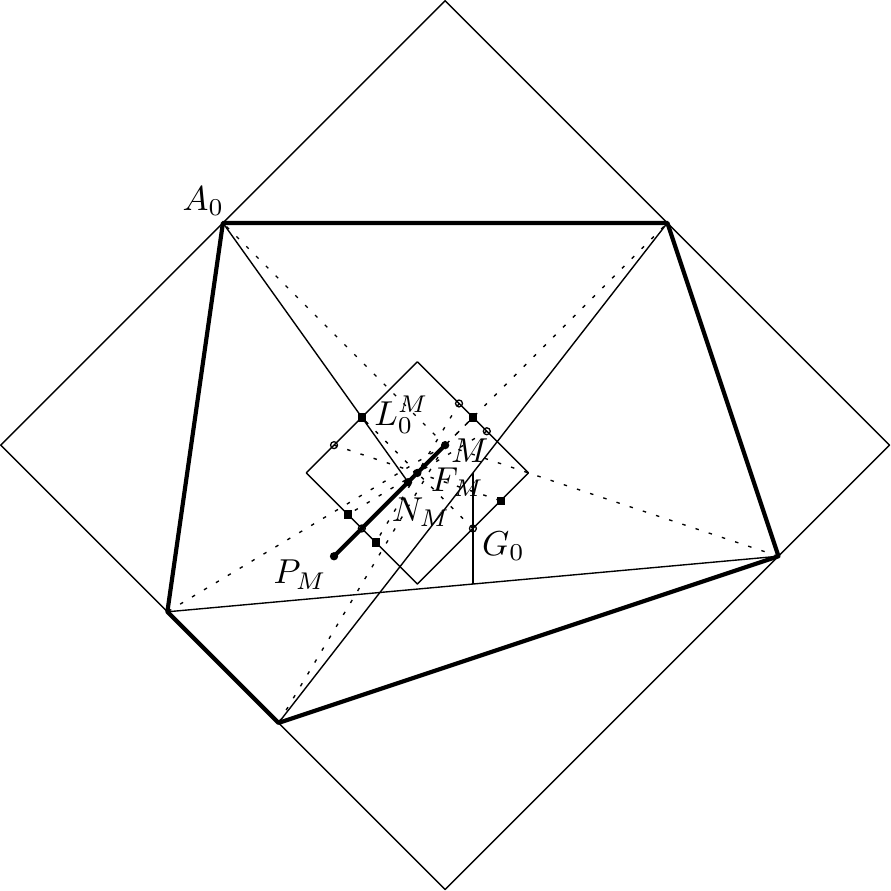}}
\end{center}
\caption{Comparison of different definitions of the Feuerbach circle for a pentagon in the $\ell_1$-norm. For (a) see Remark \ref{rem:feuerbach1}, and for (b) see Remark \ref{rem:feuerbach2}. Respective radii are marked by dotted line segments, and the relevant part of the Euler line $[MP_M]$ is marked in bold. One pair of special points on the Feuerbach circle is constructed in each case (with thin solid auxiliary lines).}
\label{fig:polygon}
\end{figure}

\begin{rem}
Collings \cite{c1967:cpatel} proved a variant of part (a) for the Euclidean plane and called the point $N_M$ differently, namely the \emph{orthocenter} of the polygon. In fact, Collings' orthocenter (per our definition, the Monge point $N_M$) was defined inductively, using the base case $d=2$, i.e., starting at sub-triangles of $P$, whose Monge point, complementary point, and $C$-orthocenter coincide. Note that an inductive definition of the Monge point as such necessitates that $M$ is the circumcenter at each stage of the recursion (otherwise the resulting points at each stage would not correspond to our definition of Monge point), and thus only works in the plane. In the context of $d$-simplices, we did not consider this recursion for precisely this reason (although of course, the respective lines exist in higher-dimensional space, and they are concurrent at the corresponding points!).
\end{rem}

\begin{rem}\label{rem:feuerbach2}
Part (b) was also proved for the Euclidean plane in \cite{c1967:cpatel}, and in analogy with the nine-point-circle of a triangle, the circle $S(F_M,\frac{R}{d})$ was named the (generalized) \emph{nine-point-circle}, although it was only observed to pass through the $(d+1)$ centroids $G_i$.
B. Herrera G\'omez \cite{hg2014:tfcatorcotcp} extended the statements, for example by proving (c) for the Euclidean plane, and by investigating related infinite families of circles.
\end{rem}

\section{Concluding remarks and open problems}

Solutions to questions from Elementary Geometry in normed spaces often yield an interesting tool and form the first step for attacking problems in the spirit of Discrete and Computational Geometry in such spaces (see, e.g., \cite{ams2012:medcacinpp1,ams2012:medcacinpp2} for the concepts of circumballs and minimal enclosing balls, or \cite[Section 4]{ms2004:tgomsas2} referring to bisectors as basis of an approach to Minkowskian Voronoi diagrams). And of course it is an interesting task for geometers to generalize notions like orthogonality (see \cite{ab1988:oinlsas1,ab1989:oinlsas2,amw2012:oboaioinls}), orthocentricity (cf. \cite{ag1960:otgomp,ms2007:tfcaoinp,mw2009:oosiscnp,rspr2015:oosimp}), isometries (see \cite{ms2009:riscmp,mss2014:gaoscmp}), and regularity (see \cite{ms2011:rtinp}) in absence of an inner product. In case of regularity, we may ask which figures are special, and what are useful concepts to describe their degree of symmetry in normed planes and spaces? For Minkowski spaces nothing really satisfactory is done in this direction, and it is clear that a corresponding hierarchical classification of types of simplices would yield the first step here. Thus, it would be an interesting research program to extend the generalizable parts of the concepts investigated in \cite{ehm2005:coscarfs,ehm2005:osatc,ehm2008:osab} to normed spaces: what particular types of simplices are obtained if special points of them, called "centers" (like circum- and incenters, centroids, Monge points, Fermat-Torricelli points etc.), coincide or lie, in cases where this is not typical (e.g., in case of the incenter), on the Euler line? In view of \cite{ms2009:riscmp,mss2014:gaoscmp}, a related interesting task might be the development of symmetry concepts based on Minkowskian isometries.

Another interesting point of view comes in with the field of geometric configurations which is summarized by the recent monograph \cite{g2009:copal}. Namely, the Three-Circles-Theorem and Miquel's Theorem can be successfully extended to normed planes (see \cite{ag1960:otgomp,ms2007:tfcaoinp,s2010:omtaiinp} and thus have acquired some recent popularity. Clifford's circle configuration, for circles of equal radii also called Clifford's Chain of Theorems (see \cite{z1940:kidk,m1997:ttctccaez}), is a direct generalization of the Three-Circles-Theorem and also part of the collection of theorems which nicely ties to visualizations of the Euler line and the Feuerbach circle in the spirit of descriptive geometry (see our discussion at the beginning of Section \ref{sec:polygons} above). Based on \cite{ag1960:otgomp,ms2007:tfcaoinp}, Martini and Spirova extend in \cite{ms2008:ccotiscmp} the Clifford configuration for circles of equal radii to strictly convex normed planes, and prove properties of the configuration as well as characterizations of the Euclidean plane among Minkowski planes. Using our terminology from Section \ref{sec:polygons} above, one may easily color the vertices of the parallelepiped $Q$ alternatingly red and blue, with $M'$ being blue. Then the projected blue vertices are centers of circles of the Clifford configuration, whereas the projected red vertices are in the intersection of certain subsets of the circles. Due to the successful extension of these topics to normed planes and spaces one might hope that also further configuration concepts can be generalized this way. E.g., one can check whether the comprehensive geometry of $n$-lines (which are the natural extensions of complete quadrilaterals; see Section 4 of the survey \cite{m1995:rrieg2}) and systems of circles corresponding with them contain parts which are generalizable this way.

As basic notions like isoperimetrix (see \cite[\S 4.4 and \S 5.4]{t1996:mg}) demonstrate, \emph{duality} (of norms) plays an essential role in the geometry of normed spaces. This concept should also be used in that part of Minkowski Geometry discussed here. It should be checked how far this important concept can be used to get, in correspondence with already obtained results, also ``dual results'', such that for example results on notions like ``circumball'' and ``inball'' might be dual to each other.

Finally we mention that still for the Euclidean plane there are new generalizations of notions, such as generalized Euler lines in view of so-called circumcenters of mass etc. (see \cite{tt2015:rotcom}), which could, a fortiori, also be studied for normed planes and spaces.

\bibliographystyle{plain}
\bibliography{bibliography}

\begin{flushleft}
Undine Leopold\\
Technische Universit\"at Chemnitz \\
Fakult\"at f\"ur Mathematik\\
D - 09107 Chemnitz\\
Germany\\[0.5ex]
undine.leopold@mathematik.tu-chemnitz.de\\[2ex]

Horst Martini\\
Technische Universit\"at Chemnitz \\
Fakult\"at f\"ur Mathematik\\
D - 09107 Chemnitz\\
Germany\\[0.5ex]
horst.martini@mathematik.tu-chemnitz.de
\end{flushleft}

\end{document}